\documentclass{amsart}
\usepackage{amsmath,amssymb,amsthm}
\newtheorem{thr}{Theorem}[section]
\newtheorem{lem}[thr]{Lemma}

\newtheorem{observ}[thr]{Observation}

\theoremstyle{definition}

\newtheorem{con}[thr]{Conjecture}

\theoremstyle{remark}

\numberwithin{equation}{section}

\def\K{{\mathrm K}}

\def\S{{\mathcal S}}

\def\F{{\mathbb F}}

\def\R{\mathbb{R}}

\begin{document}

\title[Sign patterns of rational matrices with large rank]{Sign patterns of rational \\ matrices with large rank}

\author{Yaroslav Shitov}
\address{National Research University Higher School of Economics, 20 Myasnitskaya Ulitsa, Moscow 101000, Russia}
\email{yaroslav-shitov@yandex.ru}

%    \subjclass is required.

\begin{abstract}
Let $A$ be a real matrix. The term rank of $A$ is the smallest number $t$ of lines
(that is, rows or columns) needed to cover all the nonzero entries of $A$. We prove
a conjecture of Li et al. stating that, if the rank of $A$ exceeds $t-3$, there is
a rational matrix with the same sign pattern and rank as those of $A$. We point out
a connection of the problem discussed with the Kapranov rank function of tropical
matrices, and we show that the statement fails to hold in general if the rank of $A$
does not exceed $t-3$.
\end{abstract}

\maketitle

\section{Introduction}

The problem of constructing a matrix over a given ordered field with specified sign pattern and rank
deserved a significant amount of attention in recent publications, see~\cite{LiEtAl} and references therein.
The present paper establishes a connection of this problem with that of computing certain rank functions
arisen from tropical geometry. We prove the conjecture on sign patterns of rational matrices formulated
in~\cite{LiEtAl}, and we present the examples showing the optimality of our result.

\section{Preliminaries}

The following notation is used throughout our paper. By $U^{m\times n}$ we denote the set of all $m$-by-$n$
matrices with entries from a set $U$, by $A_{ij}\in U$ we denote an entry of a matrix $A\in U^{m\times n}$.
By $U_{(i)}$ we denote the $i$th row of $U$, and we call a \textit{line} of a matrix any of its columns or rows.

A field $R$ is called \textit{ordered} if, for some subset $P\subset R$ closed under addition and multiplication,
the sets $P$, $-P$, and $\{0\}$ form a partition of $R$. The elements of $P$ are then called \textit{positive}, and
those from $-P$ \textit{negative}. The \textit{sign pattern} of a matrix $A\in R^{m\times n}$ is the matrix $S=\S(A)\in\{+,-,0\}^{m\times n}$
defined as $S_{ij}=+$ if $A_{ij}$ is positive, $S_{ij}=-$ if $A_{ij}$ is negative, and $S_{ij}=0$ if $A_{ij}=0$.
The \textit{minimum rank} of a sign pattern $S$ with respect to $R$ is the minimum of the ranks of matrices $B$ over $R$
satisfying $\S(B)=S$.

There are a significant number of recent publications devoted to the study of the minimal ranks of sign patterns
(see~\cite{LiEtAl} and references therein), and our paper aims to prove a conjecture formulated in~\cite{LiEtAl}.
This conjecture relates the minimal rank of a pattern with a concept of the \textit{term rank} of a matrix, which is defined
as the smallest number of lines needed to include all the nonzero elements of that matrix.
The classical \textit{K\"{o}nig's theorem} states the the term rank of a matrix $A$ equals the maximum number of nonzero
entries of $A$ no two of which belong to the same line, so the term rank of a sign pattern $S$ can be thought
of as the maximum of the ranks of matrices $C$ over $R$ satisfying $\S(C)=S$. Now we can formulate the conjecture
by Li et al.\,\,relating the concepts of minimum and term ranks for sign pattern matrices.

\begin{con}\cite[Conjecture 4.2]{LiEtAl}\label{conLiEtAl}
Assume $S$ is a sign pattern matrix with term rank equal to $t$, and let $r$ be the minimum rank of $S$ over the reals.
If $r\geqslant t-2$, then the minimum rank of $S$ over the rationals is $r$ as well.
\end{con}

In Section~3 we develop a combinatorial technique which allows to prove Conjecture~\ref{conLiEtAl}. In Section~4
we establish the connection of the problem discussed with the Kapranov rank function of Boolean matrices introduced
in~\cite{DSS}. We also make the use of matroid theory to prove the optimality of the bound in Conjecture~\ref{conLiEtAl}
by showing that its statement fails to hold in general if $r$ is less than $t-2$.

\section{Proof of the result}

We start with two easy observations helpful for further considerations.

\begin{observ}\label{observsign2}
Multiplying a row of a real matrix $A$ by a nonzero number
will not change the minimal ranks of its sign pattern.
\end{observ}

\begin{proof}
Trivial.
\end{proof}

\begin{observ}\label{observsign1}
Let $r$ and $t$ be, respectively, the minimum and maximum ranks of a sign pattern $S$ with respect to an ordered field $R$.
Then, for any integer $h\in[r,t]$, there is a matrix over $R$ which has rank $h$ and sign pattern $S$.
\end{observ}

\begin{proof}
%Follows from the fact that the rank of a matrix obtained by changing a single entry either equals or differs by one from that of the initial matrix.
Changing a single entry produces a matrix whose rank differs by at most $1$ from that of the initial matrix.
\end{proof}

The following lemma gives a useful description of the rank of a block
matrix. We say that a linear subspace $S\subset\R^d$ is \textit{rational}
if $S$ has a basis consisting of vectors that have rational coordinates
only.

\begin{lem}\label{lemsignblock}
Let $V_1\in\mathbb{Q}^{p\times(p-1)}$ and $V_2\in\mathbb{Q}^{(q-1)\times q}$
be rational matrices that have ranks $p-1$ and $q-1$, respectively.
Then the set $\mathcal{W}$ of all $W\in\R^{p\times q}$ for which the matrix
$U=\left(\begin{smallmatrix} W & V_1 \\ V_2 & O \end{smallmatrix}\right)$
has rank $p+q-2$ is a rational subspace.
\end{lem}

\begin{proof}
Note that rational elementary transformations on the first $p$ rows or first $q$ columns
of $U$ can not break the property of $\mathcal{W}$ to be a rational subspace. So we can
assume that $V_1$ and $V_2$ differ from the identity matrices by adding the zero column
and row, which case is easy.
\end{proof}

%Let a matrix $U\in\R^{t\times t}$ have the form $\left(\begin{smallmatrix} W & V_1 \\ V_2 & O \end{smallmatrix}\right)$,
%where $O$ is the zero matrix, and matrices $V_1\in\R^{p\times(p-1)}$ and $V_2\in\R^{(q-1)\times q}$ are rational and have
%ranks $p-1$ and $q-1$, respectively. Then the set of all $W$

Now we are ready to prove Conjecture~\ref{conLiEtAl} in a special case.

%\begin{lem}\label{lemsign1}
%Let $A$ be a real $n$-by-$m$ matrix of rank $n-2$. If every column of $A$ contains at least three
%nonzero entries, then there is a rational $n$-by-$m$ matrix which has rank $n-2$ and sign pattern equal to that of $A$.
%\end{lem}
%
%\begin{proof}
%
%\end{proof}
%
%The following lemma is devoted to a more general case.

\begin{lem}\label{lemsign2}
For any real $m$-by-$n$ matrix $A$ of rank $n-2$, there is a rational $m$-by-$n$ matrix which has rank $n-2$ and sign pattern equal to that of $A$.
\end{lem}

\begin{proof}
By the assumptions, there is a rank-two matrix $B\in\R^{n\times2}$ for which the matrix $AB$ is zero. Observation~\ref{observsign2}
allows one to assume that the first column of $B$ consists of zeros and ones. Let $X$ be a matrix whose $(i,j)$th
entry is a variable if $A_{ij}\neq0$ and $X_{ij}=0$ otherwise.

For a sufficiently large integer $N>0$, we set $C_{jk}=[NB_{jk}]/N$. Note that, for every row index $i$,
the matrix formed by the rows of $B$ with indexes $j$ satisfying $A_{ij}\neq 0$ has the same rank as
the matrix formed by the rows of $C$ with the same indexes. For every $i$, we assign to every free variable $X_{ig}$ of the
linear system $X_{(i)}C=(0\,0)$ the value $[NA_{ig}]/N$. Solving those systems, we get as a solution a rational
matrix $X=X(N)$ which satisfies $XC=0$. Since $X(N)\rightarrow A$ as $N\rightarrow\infty$, the matrices $X(N)$ and $A$
have the same sign pattern for sufficiently large $N$.
%and then it suffices to construct a rational matrix $Q$ which satisfies $QC=0$ and has a sign pattern equal to that
%of $A$.
%
%Denoting, for every row index $i$, by $Z_i$ the set of all $j$ satisfying $A_{ij}\neq 0$, we have the three possibilities.
%Denoting, for every row index $i$, by $B_i$ the matrix formed by the columns of $B$ with indexes satisfying $A_{ij}\neq 0$, we have the three possibilities.
%
%1. If $B_i$ is zero, then the $i$th row of $QC$ is zero independently of the choice of $Q$.
%
%2. Assume that $B_i$
\end{proof}

Now let us prove the key result of the section.

\begin{thr}\label{thrmaintermrealis0}
Let $A$ be a real matrix with term rank equal to $t$. If the rank of $A$ equals $t-2$, then there is
a rational matrix which has rank $t-2$ and the same sign pattern as $A$.
\end{thr}

\begin{proof}
1. Up to row and column permutations, $A$ is an $n$-by-$m$ matrix of the form $\left(\begin{smallmatrix} B & C \\ D & O \end{smallmatrix}\right)$,
where the matrix $B\in\R^{p\times q}$ satisfies $p+q=t$, and $O$ is the zero matrix. If the rank of $D$ is less than $q-1$, then by Lemma~\ref{lemsign2}
we can construct a rational matrix $D'$ of rank $q-2$ with the same sign pattern as $D$. Choosing $B'$ and $C'$ as arbitrary matrices with
sign patterns equal to those of $B$ and $C$, we get that the rank of $\left(\begin{smallmatrix} B' & C' \\ D' & O \end{smallmatrix}\right)$ is
at most $t-2$, and we are done. We can assume in what follows that $D$ has rank at least $q-1$ and, similarly, that $C$ has rank at least $p-1$.
Since the rank of $A$ is $t-2=p+q-2$, we conclude that the rank of $D$ exactly equals $q-1$ and the rank of $C$ is $p-1$.

2. By Step~1, the rows of $C$ are linearly dependent, and we can assume by Observation~\ref{observsign2} that the coefficients of this
linear dependence are rational. In other words, the columns of $C$ generate a rational subspace in $\R^p$. Since rational points are
dense in rational subspaces, we can assume that the matrix $C$ (and the matrix $D$, similarly) consists of rational numbers, in which case
the result follows from Lemma~\ref{lemsignblock}.
\end{proof}

Now we are ready to prove Conjecture~\ref{conLiEtAl}.

\begin{thr}\label{thrmaintermrealis}
Let $A$ be a real matrix with term rank equal to $t$. If the rank of $A$ is at least $t-2$, then there is
a rational matrix which has the same sign pattern and rank as those of $A$.
\end{thr}

\begin{proof}
Note that adding a repeating row does not affect the rank of a matrix, and the term rank of the matrix obtained
is either equal to or greater by one than that of the initial matrix. Therefore, adding a sufficient number of
repeating rows to $A$, we get a matrix $A'$ satisfying the assumptions of Theorem~\ref{thrmaintermrealis0}.
So we can find a rational matrix $B$ which has the same sign pattern as that of $A$ and rank not exceeding the rank of $A$.
Now the result follows from Observation~\ref{observsign1}.
\end{proof}

\section{Optimality of the result}

To construct sign patterns of term rank $t$ realizable by real matrices of rank $t-3$ but
not by rational matrices of that rank, we need to recall the definition of another rank concept.
For $\F$ a field, define the \textit{Kapranov rank} of a matrix $B\in\{0,1\}^{m\times n}$
with respect to $\F$ as the smallest possible rank of a matrix $C\in\F^{m\times n}$ satisfying
$C_{ij}=0$ if and only if $B_{ij}=0$. The following lemma points out a connection between the
quantity introduced (which we denote by $\K_\F(B)$ in what follows) and the problem of pattern realisability.

\begin{lem}\label{lemKapreal}
Assume $R_1$ is an ordered field, and a matrix $B\in\{0,1\}^{m\times n}$
satisfies $r=\K_{R_1}(B)<\K_{R_2}(B)$ for any field $R_2$ strictly contained in $R_1$.
Then there is a sign pattern $S\in\{0,+,-\}^{m\times n}$ realizable by a matrix over $R_1$
of rank $r$ but not by a matrix over $R_2$ of that rank.
\end{lem}

\begin{proof}
By definition of Kapranov rank, there is a matrix $A\in{R_1}^{m\times n}$ which has rank $r$
and satisfies $A_{ij}=0$ if and only if $B_{ij}=0$. Denoting the sign pattern of $A$ by $S$,
one can see that $S$ is not realizable by a matrix over $R_2$ of rank $r$.
\end{proof}

Now we see that a sign pattern realizable over $R_1$ but not over $R_2$
always exists if we have a zero-one matrix whose Kapranov rank over $R_2$
is greater than that over $R_1$. It turns out that producing zero-one matrices
with this property can be performed by the use of matroid theory, and let us
recall the basic definitions of this theory~\cite{Ox}. A \textit{matroid} $M$ on a finite
set $E$ is defined by the set $\mathcal{B}\subset 2^E$ of its \textit{bases},
which are supposed to satisfy the following conditions: (1) $\mathcal{B}\neq\varnothing$;
(2) if $A,B\in \mathcal{B}$ and $a\in A\setminus B$, then there is a $b\in B\setminus A$
such that $A\setminus\{a\}\cup\{b\}\in I$. All the bases can easily be shown to have
the same cardinality, and this cardinality is called the \textit{rank} of a matroid $M$.
A \textit{circuit} in $M$ is a minimal set which is a subset of no $B\in\mathcal{B}$.
A \textit{dual matroid} $M^*$ has as its bases the complements of the bases of $M$, and
a circuit in $M^*$ is called a \textit{cocircuit} for $M$.
The matroid $M$ is \textit{representable} over a field $\F$ if we can assign vectors from $\F^d$
to the elements of $E$ in such a way that a set $B$ is a basis of the linear span of $E$ if and only
if $B\in\mathcal{B}$. Finally, define a cocircuit matrix $\mathcal{C}=\mathcal{C}_M$ of $M$ as a matrix with
rows indexed by elements of $E$ and columns indexed by cocircuits such that $\mathcal{C}_{ij}=1$ if $i$ belongs
to the $j$th cocircuit and $\mathcal{C}_{ij}=0$ otherwise. The following theorem allows one to construct
matrices whose Kapranov rank depends on a ground field.

\begin{thr}\cite[Proposition~7.2 and Theorem~7.3]{DSS}\label{thrKapMatr}
If $M$ is a matroid of rank $r$ and $\mathcal{C}$ its cocircuit matrix, then $\K_\F(\mathcal{C})\geqslant r$ for any field $\F$.
For $\F$ infinite, the condition $\K_\F(\mathcal{C})=r$ holds if and only if $M$ is representable over $\F$.
\end{thr}

The following well-known fact connects the notions of matroid duality and representability.

\begin{thr}\cite{Ox}\label{thrmatrdual}
If a matroid $M$ is representable over a field $\F$, then so is its dual $M^*$.
\end{thr}

Note that the matroid duality is an involution, that is, the condition $(M^*)^*=M$ holds.
This shows that Theorem~\ref{thrmatrdual} holds as well in the opposite direction. We also
need the classical example of a non-representable matroid, which appeared in a foundational
paper by Saunders MacLane~\cite{MacLane}.

\begin{thr}\cite[Theorem~3]{MacLane}\label{rank3repres}
Let $K$ be a finite algebraic field over the field of rational numbers. Then there exists a matroid $M$ of rank $3$ which
is representable over $K$ but over no field strictly contained in $K$.
\end{thr}

Now we are ready to prove the theorem stating that the bound of $t-2$ is optimal in Theorem~\ref{thrmaintermrealis}.

\begin{thr}
Let $K$ be an ordered finite algebraic field over the field of rational numbers.
Then there exists a matrix $A\in K^{n\times m}$ of rank $n-3$ with the following
property: the entries of any matrix $A'\in K^{n\times m}$ which has the same rank
and sign pattern as those of $A$ generate the whole field $K$.
\end{thr}

\begin{proof}
Using Theorem~\ref{rank3repres}, we get a rank-three matroid $M$ representable over $K$ but
not over any field strictly contained in $K$. Denoting the number of its vertices by $n$,
we see that by definitions the dual $M^*$ has rank $n-3$. Then we use Theorem~\ref{thrmatrdual}
to conclude that $M^*$ is representable over $K$ but not over any field strictly contained in $K$.
From Theorem~\ref{thrKapMatr} it follows that the cocircuit matrix of $M^*$, which has $n$ rows,
has also Kapranov rank $n-3$ with respect to $K$ and greater Kapranov rank with respect to any
field strictly contained in $K$. Application of Lemma~\ref{lemKapreal} now completes the proof.
\end{proof}

\end{document}